\documentclass[a4paper,11pt,reqno]{amsart}
\usepackage{amsthm}
\usepackage{amsmath}
\usepackage{enumerate}
\usepackage{amsfonts}
\usepackage{amssymb}
\usepackage{fullpage}
\usepackage{amsmath,amscd}
\usepackage{stmaryrd}
\usepackage{graphicx}
\usepackage{array}
\usepackage{amsmath,amssymb,amsfonts,dsfont}
\usepackage[utf8]{inputenc} 
\usepackage[english]{babel}
\usepackage[T1]{fontenc} 
\usepackage{graphicx}
\usepackage{float}
\usepackage{pdfpages}
\usepackage[a4paper, margin = 3cm, bottom = 3cm]{geometry}
\usepackage{ifpdf}
\usepackage{marginnote}
\usepackage{hyperref}
\usepackage{tikz}
\hypersetup{pdfborder=0 0 0, 
	    colorlinks=true,
	    citecolor=black,
	    linkcolor=blue,
	    urlcolor=red,
	    pdfauthor={Guillaume Tahar}
	   }
\newtheorem{thm}{Theorem}[section]

\newtheorem{prop}[thm]{Proposition}
\newtheorem{lem}[thm]{Lemma}
\theoremstyle{definition}
\newtheorem{defn}[thm]{Definition}
\theoremstyle{remark}
\newtheorem{rem}[thm]{Remark}
\theoremstyle{definition}

\theoremstyle{definition}

\theoremstyle{definition}

\numberwithin{equation}{section} 
\title{A topological bound on the Cantor-Bendixson rank of meromorphic differentials}
\author{Guillaume Tahar} 
\address[Guillaume Tahar]{Faculty of Mathematics and Computer Science, Weizmann Institute of Science,
Rehovot, 7610001, Israel}
\email{tahar.guillaume@weizmann.ac.il}

\date{July 23, 2020}
\begin{document}
\begin{abstract}
In translation surfaces of finite area (corresponding to holomorphic differentials), directions of saddle connections are dense in the unit circle. On the contrary, saddle connections are fewer in translation surfaces with poles (corresponding to meromorphic differentials). The Cantor-Bendixson rank of their set of directions is a measure of descriptive set-theoretic complexity. Drawing on a previous work of David Aulicino, we prove a sharp upper bound that depends only on the genus of the underlying topological surface. The proof uses a new geometric lemma stating that in a sequence of three nested invariant subsurfaces the genus of the third one is always bigger than the genus of the first one.
\end{abstract}
\maketitle
\setcounter{tocdepth}{1}
\tableofcontents

\section{Introduction}

In \cite{BS}, Bridgeland and Smith give a connection between spaces of stability conditions on triangulated categories and moduli spaces of quadratic meromorphic differentials. Slopes of stable objects in triangulated categories correspond to directions of saddle connections (geodesic segments between singularities) for the flat structure defined by a meromorphic differential on a Riemann surface. They appear in counting BPS states of string theory, see \cite{KNTZ}.\newline
The set of directions of saddle connections for a meromorphic quadratic differential crucially depends on the order of the singularities. It is well known (see \cite{Zo}) that when poles are at most simple, the flat surface is of finite area and the directions of saddle connections are dense in the unit circle. On the contrary, when there is a pole of order at least two, then the flat surface is of infinite area and the set of directions of saddle connections is closed, see \cite{Au,Ta}.\newline
In order to quantify the complexity of the set of directions of saddle connections, Aulicino introduces in \cite{Au} the Cantor-Bendixson rank of a meromorphic quadratic differential that is the smallest number of times the set of directions has to be derived to stabilize (see Definition 2.1). The moduli space is stratified by fixing the orders of the singularities. For differentials that belong to a given stratum, the rank is bounded by the dimension of the stratum.\newline
Actually, the framework used by Aulicino is not exactly that of meromorphic differentials but rather holomorphic differentials with a slit. He considers only saddle connections that do not cross the slit. Up to easy surgeries, it is always possible to associate to a meromorphic differential a holomorphic differential with slits in such a way that the set of directions of saddle connections is the same, see \cite{Au}.
In this paper, we provide a sharp bound on the Cantor-Bendixson rank. The bound provided in \cite{Au} was the dimension $2g+n-1$ of a stratum $\mathcal{H}(a_1,\dots,a_n)$ of holomorphic differentials with $n$ zeroes and slits on a surface of genus $g$. Instead, we prove that the maximal rank is $2g$, $2g+1$ or $2g+2$ depending on the number of zeroes and poles. It is a significant improvement in particular in principal strata (where the number $n$ of zeroes is maximal).\newline
In addition to the counting of BPS states in some supersymmetric field theories, the Cantor-Bendixson rank of a meromorphic differential is involved in the counting of saddle connections. For a translation surface $X$, let $N(L,X)$ be the number of saddle connections of length smaller than $L$. The classical result of Masur (see \cite{Ma}) state in particular that $N(L,X) \sim L^{2}$ when $X$ is defined by a holomorphic $1$-form. Following an unpublished conjecture of Aulicino, Pan and Su, when $X$ is defined by a meromorphic $1$-form (with at least one pole) whose Cantor-Bendixson rank is $k_{X}$, then we have:\newline
$N(L,X) \sim L^{0}$ if $k_{X}=1$\newline
$N(L,X) \sim L(log~L)^{k_{X}-2}$ if $k_{X} \geq 2$.\newline
An optimal estimate of the Cantor-Bendixson rank of meromorphic differentials thus takes a greater importance.\newline

\section{Statement of main results}

In the following, we denote by \textit{translation surface with pole} a pair $(X,\phi)$ where $X$ is a compact Riemann Surface of genus $g$ and $\phi$ is a meromorphic $1$-form whose singularities are of order $a_1,\dots,a_n,-b_1,\dots,-b_p$. Such pairs belong to strata $\mathcal{H}(a_1,\dots,a_n,-b_1,\dots,-b_p)$ where $\sum \limits_{i=1}^n a_i - \sum \limits_{j=1}^p b_j = 2g-2$. We assume the the meromorphic differential has at least one zero and one pole.\newline

The meromorphic differential defines a flat metric on the surface punctured at the poles. The zeroes of the differential are the conical singularities of the metric, see Section 3 for details. Saddle connections are geodesic segments whose ends are conical singularities. We study the set $\Theta(X,\phi) \subset \mathcal{S}^{1}$ of directions of saddle connections where $\mathcal{S}^{1}$ is the unit circle. When there is no ambiguity, we simply denote this set by $\Theta$.

We introduce the definition of Cantor-Bendixson rank used by Aulicino in \cite{Au}.

\begin{defn}
Let $A \subset \mathcal{S}^{1}$. The derived set $A^{\star}$ is the subset of $A$ such that all isolated points of $A$ are removed. We denote by $A^{\star n}$ the $n^{th}$ derived set of $A$.\newline
The Cantor-Bendixson rank of a set $A$ is the smallest non-negative integer $n$ such that $A^{\star n+1} = A^{\star n}$.\newline
The Cantor-Bendixson rank of a translation surface with poles is the Cantor-Bendixson rank of its set of directions of saddle connections $\Theta(X,\phi)$.
\end{defn}

The main result of the paper is a topological upper bound on the Cantor-Bendixson rank of the set of directions of saddle connections for a meromorphic differential of given genus. The equality case is realized by some surfaces constructed in \cite{Au}. Our main tool in proving Theorem 2.2 is Lemma 4.3 which states that in a sequence of three nested invariant subsurfaces, the genus always increases. This technical lemma can be used for general translation surfaces (and not only meromorphic differentials).

\begin{thm}
For a meromorphic differential of genus $g$, the Cantor-Bendixson rank of the set of directions of saddle connections is at most $2g+2$.\newline
If the differential has at most one zero or at most one pole, the bound can be improved to $2g+1$. Furthermore, if the differential has exactly one zero and one pole, the bound can be improved to $2g$.
\end{thm}

The structure of the paper is the following: \newline
- In Section 3, we recall the background about translation surfaces: flat metric, saddle connections, moduli space, core, directional foliation.\newline
- In Section 4, drawing on an analysis of the relation between the Cantor-Bendixson rank and the topological complexity of invariant components, we prove the main theorem.

\section{Translation structures defined by meromorphic differentials}

\subsection{Translation structures}

Let $X$ be a compact Riemann surface and let $\phi$ be a meromorphic $1$-form. We denote by $\Lambda$ the set of zeroes of $\phi$ and by $\Delta$ the set of its poles.\newline

Outside $\Lambda$ and $\Delta$, integration of $\phi$ gives local coordinates whose transition maps are of the type $z \mapsto z+c$. The pair $(X,\phi)$ seen as a compact surface with such an atlas is called a \textit{translation surface with poles}.\newline

In a neighborhood of a zero of order $a>0$, the metric induced by $\phi$ admits a conical singularity of angle $(1+a)2\pi$, see \cite{Zo} for details.

Geometry of neighborhoods of poles is not useful in our study because saddle connections belong to the core of the translation surface which is precisely the complement of a union of neighborhoods of the poles.

\subsection{Moduli space}

If $(X,\phi)$ and $(X',\phi')$ are translation surfaces such that there is a biholomorphism $f$ from $X$ to $X'$ such that $\phi$ is the pullback of $\phi'$, then $f$ is an isometry for the flat metrics defined by $\phi$ and $\phi'$.\newline
We define the moduli space of meromorphic differentials as the space of equivalence classes of translation surfaces with poles $(X,\phi)$ up to biholomorphism preserving the differential.\newline
We denote by $\mathcal{H}^{k}(a_1,\dots,a_n,-b_1,\dots,-b_p)$ the \textit{stratum} that corresponds to meromorphic $1$-forms with singularities of orders $a_1,\dots,a_n,-b_1,\dots,-b_p$. The integer $g$ is the genus of the underlying Riemann surface. For sake of simplicity we will use the term of strata of genus $g$ to designate strata of differentials that exist only on surfaces of genus $g$.\newline

\begin{defn} A saddle connection is a geodesic segment joining two conical singularities of the flat surface such that all interior points are not conical singularities.
\end{defn}

We associate homology classes of $H_{1}(\widetilde{X}\setminus\Delta,\Lambda)$ to saddle connections. Two saddle connections are said to be parallel when their relative homology classes are linearly dependant over $\mathbb{R}$.\newline
The holonomy vector of a saddle connection is the period of its relative homology class. Its direction is the argument of the period and its length is the modulus of the period. Strata are complex-analytic orbifolds with local coordinates given by the period map, see \cite{BS}.\newline

\subsection{Core of a translation surface with poles}

Nearly all of the geometry of a translation surface with poles is encompassed in a subsurface of finite area that is the convex hull of the conical singularities. This notion of core of a translation surface with poles was introduced in \cite{HKK} and developed in \cite{Ta}. Most foundational results about the core given in this subsection are proved in the latter paper.\newline
In \cite{Au}, Aulicino considers \textit{slit translation surfaces} that are translation surfaces (of finite area) with distinguished cuts that play the role of the boundary of the core. For that reason, our foundational results agree.

\begin{defn} A subset $E$ of a translation surface with poles $(X,\phi)$ is \textit{convex} if and only if every element of any geodesic segment between two points of $E$ belongs to $E$.\newline
The convex hull of a subset $F$ of a translation surface with poles $(X,\phi)$ is the smallest closed convex subset of $X$ containing $F$.\newline
The core of $(X,\phi)$ is the convex hull $core(X)$ of the conical singularities $\Lambda$ of the meromorphic differential.\newline
$\mathcal{I}\mathcal{C}(X)$ is the interior of $core(X)$ in $X$ and $\partial\mathcal{C}(X) = core(X)\ \backslash\ \mathcal{I}\mathcal{C}(X)$ is its boundary.
\end{defn}

The core separates the poles from each other. The following lemma shows that the complement of the core has as many connected components as there are poles. We refer to these connected components as \textit{domains of poles}. It has been proved as Proposition 4.4 and Lemma 4.5 in \cite{Ta}.

\begin{prop} For translation surface with $p$ poles $(X,\phi)$, $\partial\mathcal{C}(X)$ is a finite union of saddle connections. Moreover, $X \setminus core(X)$ has $p$ connected components. Each of them is a topological disk that contains a unique pole.
\end{prop}

\begin{rem}
We can define a Cantor-Bendixson rank for every connected component of the interior of the core of the surface. Since strata are decomposed into chambers depending on the topological pair $(X,core(X))$, we can deduce specific optimal ranks for chambers.
\end{rem}

\subsection{Directional Foliation}

Dynamics of translation surfaces with poles are different from translations of finite area because most trajectories go to infinity. The following proposition describes the invariant components of the directional foliation in a translation surface with poles. The following result has been proved as Proposition 5.5 in \cite{Ta}.

\begin{prop} Let $(X,\phi)$ be a translation surface with poles. Cutting along all saddle connections sharing a given direction $\theta$, we obtain finitely many connected components called \textit{invariant components} of the directional foliation in direction $\theta$. 
There are four types of invariant components: \newline
- \textbf{finite area cylinders} where the leaves (waist curves) are periodic with the same period,\newline
- \textbf{minimal components} of finite area where the foliation is minimal and whose dynamics are given by a nontrivial interval exchange map,\newline
- \textbf{infinite area cylinders} bounding a simple pole and where the leaves are periodic with the same period,\newline
- \textbf{free components} of infinite area where generic leaves go from a pole to another or return to the same pole.\newline
\newline
Finite area components belong to $core(X)$.
\end{prop}

The simplest example of a minimal component in a translation surface is a flat torus foliated in an irrational direction. Its dynamics is given by an irrational rotation.\newline

A first step in the interpretation of the set $\Theta(X,\phi)$ of directions of saddle connections in terms of invariant components is the following proposition, proved as Proposition 5.10 in \cite{Ta} and Lemma 4.4 in \cite{Au}.

\begin{prop} Let $(X,\phi)$ be a translation surface with poles. Then the directions of the waist curves of finite area cylinders and the directions of minimal components of $(X,\phi)$ are exactly the accumulation points of the set $\Theta(X,\phi)$ of directions of saddle connections of $(X,\phi)$. In particular, $\Theta(X,\phi)$ is closed in the unit circle $\mathcal{S}^{1}$.
\end{prop}

\section{Invariant components and Cantor-Bendixson rank}

In this section, we consider only invariant components of finite area because they are those which contribute to the set of saddle connections (they belong to the core of the surface). For sake of simplicity, we use the term \textit{invariant component} only in the case of invariant components of finite area.

\subsection{$\omega$-limit sets}

In the following, we use the notion of $\omega$-limit set to describe how a sequence of invariant components accumulates on another. This approach was already used by Aulicino in \cite{Au}.

\begin{defn}
In a translation surface with poles $(X,\phi)$, we consider a sequence $(B_{k}(\theta_{k}))_{k \in \mathbb{N}}$ of invariant components. The $\omega$-limit set of sequence $(B_{k})_{k \in \mathbb{N}}$ is $\Omega = \bigcap\limits_{n=1}^{\infty} \overline{
\bigcup\limits_{k\geq n} B_{k}}$.
\end{defn}

The following lemma corresponds to Lemma 7.2 and Lemma 7.3 in \cite{Au}. We reformulate its proof in the context of translation surfaces with poles.

\begin{lem}
In a translation surface $X$, we consider a sequence $(B_{k}(\theta_{k}))_{k \in \mathbb{N}}$ of invariant components of direction $\theta_{k}$ where $\theta_{k} \longrightarrow \theta$. Its $\omega$-limit set $\Omega$ is a union of invariant components that belong to the same direction $\theta$. Moreover, there is a subsequence $\alpha$ such that any component $B_{\alpha(k)}$ is included in $\Omega$.
\end{lem}

\begin{proof}
First, $\Omega$ is clearly a closed subset of $core(X)$ that contains at least one conical singularity. At least one trajectory in direction $\theta$ starting from this singularity also belongs to $\Omega$ because it is contained in the closure of an infinite sequence of invariant components whose direction converges to $\theta$. If the trajectory goes to a pole, we get a contradiction because such a trajectory would be stable under small perturbations and in the sequence there could not be invariant components whose direction would lie in a neighborhood of $\theta$. Therefore, either the trajectory is critical (ends in another conical singularity) or it is minimal and accumulates on a minimal invariant component whose boundary is formed by saddle connections of direction $\theta$. We proceed this way for every point of $\Omega$. Therefore, $\Omega$ is a union of saddle connections and invariant components in direction $\theta$. Following Proposition 3.5, there is only a finite number of invariant components in a given direction. Therefore, $\Omega$ is a finite union of invariant components and saddle connections.\newline
By contradiction, we assume there is a number $N$ such that $\forall n \geq N$, $B_{n}$ is not included in $\Omega$. Then there is a saddle connection $\sigma$ of the boundary of $\Omega$ such that there is a subsequence $\beta$ such that every component $B_{\beta(n)}$ crosses $\sigma$. We also assume $\theta_{\beta(n)}$ converges monotonically to $\theta$. Saddle connection $\sigma$ goes from a conical singularity $z_{0}$ to another conical singularity $z_{1}$. We assume $\sigma$ is a left boundary of $\Omega$. Turning clockwise around $z_{1}$ with an angle of $\pi$, we get another direction that belongs to the closure of components $B_{\beta(n)}$. Using the same argument as previously, we show that this direction is either minimal or critical. In the first case, we get another component in direction $\theta$ that would necessarily belong to $\Omega$ and that would be located precisely at the left of $\sigma$. Therefore, the direction is critical. We get a saddle connection from $z_{1}$ to another saddle connection $z_{2}$. Turning clockwise around $z_{2}$ with an angle of $\pi$, we get another direction to which we can apply the same reasoning. This process ends with a cyclic chain of saddle connections bounding a cylinder to the left of $\sigma$ whose periodic geodesics belong to direction $\theta$. We choose a closed geodesic $\gamma$ of the cylinder. It is crossed by every component $B_{\beta(n)}$. There is at least one point $z$ of $\gamma$ such that any neighborhood of $z$ intersects an infinite number of components $B_{\beta(n)}$. Such a point $z$ thus belongs to $\Omega$ and the whole cylinder is included in $\Omega$ too. This contradicts the hypothesis. Therefore, there is always an infinite number of components of the sequence that always remain in $\Omega$.
\end{proof}

\subsection{Topological upper bound}

We first prove that in a sequence of inclusions of three invariant subsurfaces (union of invariant components in the same direction), the genus should increase. The intuitive idea is that between two invariant subsurfaces, there is the gluing of a pair of pants. This operation may either increase the genus or the number of connected components (depending on the way the pair of pants is glued). We prove that two consecutive such operations always increase the genus at least one time.

\begin{lem}
In a translation surface with three invariant subsurfaces $A \subset B \subset C$ such that $\theta_{A} \neq \theta_{B}$ and $\theta_{B} \neq \theta_{C}$, the genus $g_{C}$ of $C$ is strictly bigger than the genus $g_{A}$ of $A$.
\end{lem}

\begin{proof}
Since we have clearly $g_{A} \leq g_{B} \leq g_{C}$, we just have to deduce a contradiction from $g_{A} = g_{B} = g_{C}$. We consider a connected component $X$ of the symmetric difference $A \Delta C$. We deduce from the initial hypothesis that $X$ is a surface of genus zero with a boundary formed by saddle connections in directions $\theta_{A}$ and $\theta_{C}$. If $X$ has several boundary components, then exactly one of those contains saddle connections in direction $\theta_{A}$ because otherwise the union of $A$ with $X$ would amount to adding a handle (the genus would increase). We denote by $X^{1}$ this special boundary component. The other boundary components are formed by saddle connections in direction $\theta_{C}$. Without loss of generality, we assume that the latter boundary components are not the boundary of cylinders in direction $\theta_{C}$ inside $X$. In this case, we could cut out these cylinders and get a smaller surface without such cylinders that displays the same situation (inclusion of three consecutive invariant surfaces). Indeed, any trajectory in direction $\theta_{B}$ that would enter into the cylinder would leave the surface. Therefore, the cylinder is disjoint from $B$. Consequently, if there is a counterexample with cylinders, there also exists a counterexample without cylinders.\newline

Then, we consider trajectories in $X$ starting from the boundary of $A$ and that belong to direction $\theta_{B}$. In a surface of genus zero, there is no minimal trajectories. Therefore, such trajectories either hit a singularity or finally cross another boundary of the surface. Because of their direction, these trajectories clearly stay in $B$. Therefore, generically, they go from a boundary saddle connection of direction $\theta_{A}$ to another (with the opposite orientation). These trajectories describe parallelograms inside $X$.\newline
If we identify the special boundary component $X^{1}$ with a circle, choosing one trajectory in each of these parallelograms defines a chord diagram where chords do not cross each other (we remind that $X$ is of genus zero). In any such drawing, there is at least one chord that separates the disk into two connected components one of which does not contain any of the drawn chords. This implies that there is a parallelogram whose complement in $X$ has two connected components one of which has no boundary saddle connection in direction $\theta_{A}$, see Figure 1. We denote by $Y$ this connected component. It is of genus zero and if it has several boundary components, only one of them has saddle connections in both direction $\theta_{B}$ and $\theta_{C}$. We denote by $Y^{1}$ this special boundary component.\newline

\begin{figure}
\includegraphics[scale=0.3]{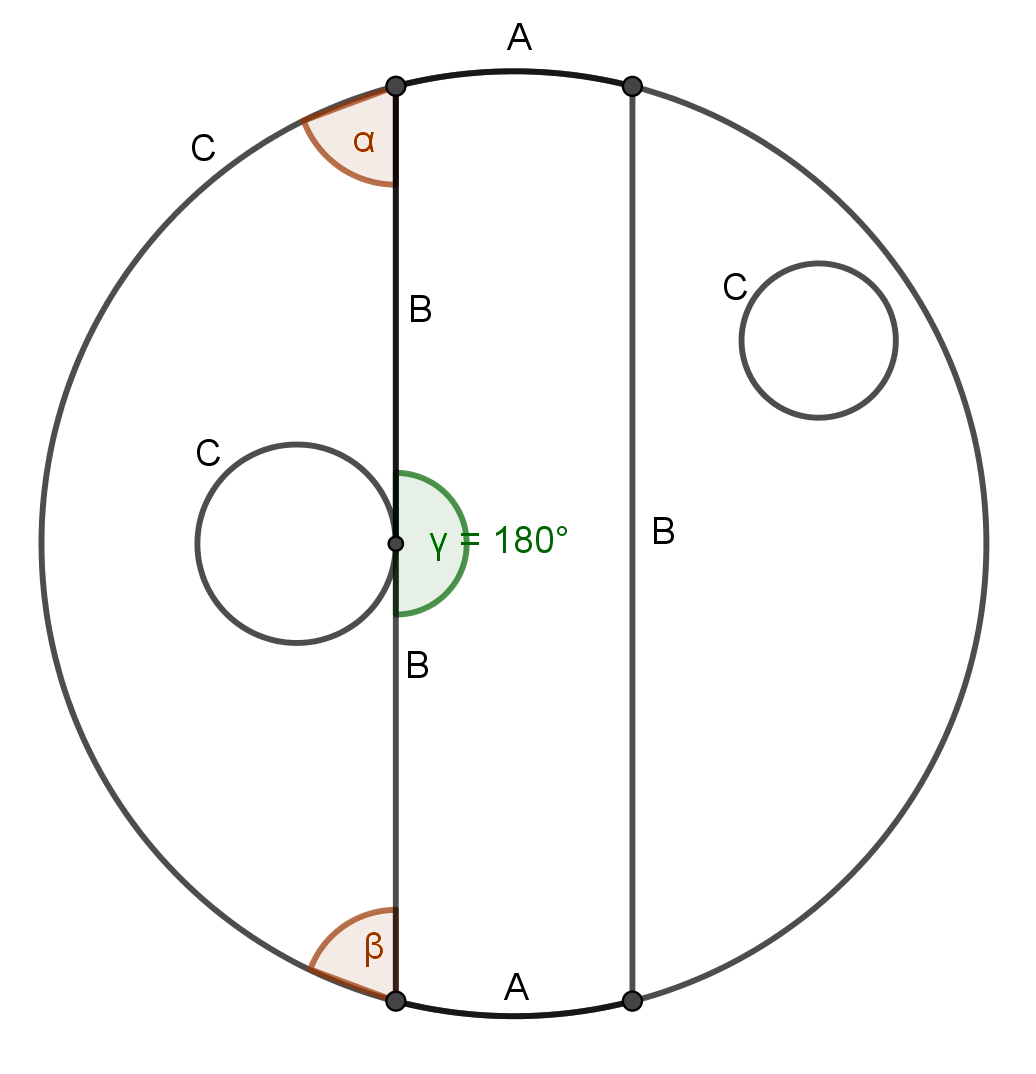}
\caption{$X$ is a surface of genus zero with three boundary components. In the middle of the disk there is a parallelogram that separates the surface into two components. The left part (in grey) is the component $Y$ we consider for the angle defect argument.}
\end{figure}

A discrete Gauss-Bonnet theorem (obtained by a geodesic triangulation of the flat surface and a counting of the triangles) implies that the sum of the angle defects in translation surface with a boundary is a topological invariant. In any corner of the boundary, if the angle of the corner is $\alpha$, then the angle defect is $\alpha-\pi$. For a singularity in the interior the surface of angle $\beta$, the angle defect is $\beta-2\pi$. In a surface of genus zero with $b$ boundary components, the sum of the angle defects is $(b-2)2\pi$. We prove first that the contribution of the boundary component $Y^{1}$ should be negative.\newline
First, if there are singularities in the interior of the surface, their angle is an integer multiple of $\pi$ so there angle defect is not negative. Corners in boundary components different from $Y^{1}$ are between saddle connections that both lie in direction $\theta_{C}$. Their angles are thus multiple integers of $\pi$. In the beginning of the proof, we excluded the case where these boundary components bound a cylinder. Therefore, their angles cannot be all equal to $\pi$. Finally, the sum of the angle defects of a boundary component is an integer multiple of $2\pi$ because otherwise the parallel transport of the translation surface along the loop would define a holonomy map which would fail to be a translation (directions are defined unambiguously in a translation surface). Consequently, for boundary components different from $Y^{1}$, the angle defect is at least $2\pi$. This implies that the total angle defect of $Y^{1}$ is at most $-2\pi$ (like the boundary of a polygon).\newline
We have to distinguish the various parts of $Y^{1}$. One part of it also belongs to $X^{1}$, it is formed by saddle connections in direction $\theta_{C}$.  Another part is formed by saddle connections in direction $\theta_{B}$ that also bound the parallelogram that separates $Y$ from the rest of $X$, see Figure 1. The last part is formed by saddle connections in direction $\theta_{C}$ that belong to other boundary components of $X$. We first focus on the latter. As we proved previously, the total angle defect around one of these boundary components is at least $2\pi$. Such a boundary component is attached to $Y^{1}$ by a unique singularity that that belongs to three corners: two inside $Y$ between saddle connections of direction $\theta_{B}$ and $\theta_{C}$, the last one inside the parallelogram (with an angle of $\pi$). To compute the contribution of this boundary component to the total angle defect of $Y^{1}$,  we have to take into account the fact that one singularity appears in two corners and also that an angle of $\pi$ does not count for $Y^{1}$ (because this corner belongs to the parallelogram). Thus, if the total angular defect of a loop around this boundary component is $2k\pi$, its contribution to $Y^{1}$ is $2(k-1)\pi$. Since we excluded cylinders, we know that $k \geq 1$. Finally, angles $\alpha$ and $\beta$ (corresponding to the two corners between the saddle connections that belong to $X^{1}$ and those that bound the parallelogram) are nonzero. Therefore, we have $\alpha-\pi > -\pi$ and$\beta-\pi> -\pi$. Since there is no negative contribution of corners between saddle connections with the same direction, the total angle defect of $Y^{1}$ is strictly bigger than $(b-2)2\pi$ where $b$ the number of boundary components. This is the contradiction we look for. This ends the proof.
\end{proof}

The upper bound proved in \cite{Au} relies on a proof by induction that directions that belong to the $(k-1)^{th}$ derived set of directions of saddle connections of a translation surface with poles are directions of invariant components whose complex dimension of the deformation space is at least $k$. We directly use the genus as measure of complexity of invariant components. 

\begin{prop}
In a translation surface with poles $(X,\phi)$, every element of $\Theta^{\star (k-1)}$ is the direction of an invariant subsurface of genus at least $\frac{k-2}{2}$. If the genus of the subsurface is exactly $\frac{k-2}{2}$, then its boundary has at least two connected components.
\end{prop}

\begin{proof}
We proceed by double induction on $k$. Every element of $\Theta^{\star}$ is the direction of an invariant component (Proposition 3.6). If it is a minimal component, its genus is at least one. If it is a cylinder, it has exactly two boundary components so the proposition holds for $k=2$.\newline
Every element of $\Theta^{\star \star}$ is the direction of an invariant subsurface that is the $\omega$-set of a sequence of invariant components (Lemma 4.2). Therefore, we have two invariant subsurfaces $A \subset B$. If the genus of $B$ is at least one, the proposition is proved. If $B$ has genus zero, then $A$ is also a subsurface of genus zero. They are both formed by cylinders. Without loss of generality, we can assume that $A$ is a cylinder. This implies that waist curves of $A$ are also waist curves of $B$. They belong to the same direction which is impossible. Therefore, the proposition holds for $k=3$.\newline

By induction, we assume the proposition holds for $k < 2m+2$. Following Lemma 4.2, every element of $\Theta^{\star (2m+1)}$ is the $\omega$-set of a sequence of invariant components whose directions lie in $\Theta^{\star (2m)}$. For the same reasons, these invariant components are $\omega$-sets of other sequences of invariant components in $\Theta^{\star (2m-1)}$. Consequently, we have three invariant components $A \subset B \subset C$ whose direction is respectively in $\Theta^{\star (2m-1)}$, $\Theta^{\star (2m)}$ and $\Theta^{\star (2m+1)}$. The genus of $A$ is at least $m-1$ and the genus of $B$ is at least $m$. Therefore, the genus of $C$ is at least $m$. If the genus of $C$ is exactly $m$, then we will prove that $C$ has at least two boundary components.\newline
If $C$ has only one boundary component, then $B$ should also have only one boundary component (otherwise, the gluing of $B \Delta C$ would increase the genus by adding a handle). The symmetric difference $B \Delta C$ is formed by connected components of genus zero. If one of these components has several boundary components, then there is a loop inside it such that the complement of this loop has two connected components, one of which contains the boundary of $C$. In this case, $B \Delta C$ is an annulus with two boundary components, each of which is formed by saddle connections in the same directions. Just like in the proof of Lemma 4.3, we compute the total angle defect of this component, where every corner has an angle which is a multiple integer of $\pi$. The total angle defect of an annulus is zero so the two boundary components bound cylinders, which is impossible. Then, we have to deal with the case where the connected components of $B \Delta C$ are polygons. The boundary of any of these polygons is formed by saddle connections in alternating directions ($\theta_{B}$ and $\theta_{C}$) with at least two of each. Drawing a path relating two of these saddle connection in direction $\theta_{B}$ and closing it inside $B$ provides a loop separating the boundary of $C$ into several connected components. Therefore, $C$ cannot have only one boundary component. The proposition holds for $k=2m+2$.\newline

Then, we assume by induction the proposition holds for $k \leq 2m+2$. Following Lemma 4.2, every element of $\Theta^{\star (2m+2)}$ is the $\omega$-set of a sequence of invariant components whose directions lie in $\Theta^{\star (2m+1)}$. Similarly, we have three invariant components $A \subset B \subset C$ whose direction is respectively in $\Theta^{\star (2m)}$, $\Theta^{\star (2m+1)}$ and $\Theta^{\star (2m+2)}$. The genus of both $A$ and $B$ is at least $m$. Therefore, the genus of $C$ is at least $m+1$ (Lemma 4.3). This ends the proof.
\end{proof}

Now we are able to prove the main theorem of the paper: the topological bounds on the descriptive set-theoretic complexity of the set of directions of saddle connections.

\begin{proof}[Proof of Theorem 2.2]
Following Proposition 4.4, every element of $\Theta^{\star (k-1)}$ is the direction of an invariant subsurface of genus $g' \geq \frac{k-1}{2}$. Since $g'$ cannot be bigger than the genus $g$ of the whole surface, $\Theta^{\star (k-1)}$ is empty if $k>2g+1$. Thus, the sequence of derived sets stabilizes and the Cantor-Bendixson rank of $(X,\phi)$ is at most $2g+2$.\newline
If the Cantor-Bendixson rank of a meromorphic differential is exactly $2g+2$, then then $\Theta^{\star (2g+1)}$ is nonempty and there is an invariant subsurface $A$ of genus $g$ with at least two boundary components. The symmetric difference $A \Delta X$ is formed by several genus zero surfaces with exactly one boundary component formed by saddle connections in direction $\theta_{A}$. Since these topological disks are also invariant in this direction and there is no such translation surfaces, they contain poles and at least one conical singularity. Therefore, the meromorphic differential should contain at least two zeroes and two poles.\newline
If the Cantor-Bendixson rank of a translation surface with poles is exactly $2g+1$, then for the same reasons there is an invariant subsurface $A$ of genus $g$. The symmetric difference $A \Delta X$ is formed by at least one topological disk. If the meromorphic differential has only one pole and one zero, then $A \Delta X$ is just a topological disk whose boundary has a total holonomy equal to zero (the residue of the unique pole is zero). Therefore, there should be at least two different conical singularities otherwise there would not be any saddle connection in the boundary of the disk (their period would automatically be zero). Therefore, there are at least three distinct singularities.\newline
As a consequence of the above, if $n=p=1$, then the Cantor-Bendixson rank of the meromorphic differential is at most $2g$.
\end{proof}

In \cite{Au}, Aulicino already constructed families of square-tiled translation surfaces with slits. They are invariant components nested in each other. These examples show that the bound of this paper is sharp for a given genus.\newline

\textit{Acknowledgements.} This research was supported by the Israel Science Foundation (grant No.  1167/17) and has received funding from the European Research Council (ERC) under the European Union Horizon 2020 research and innovation programme (grant agreement No. 802107). The author is grateful to David Aulicino, Dmitry Novikov and the anonymous referee for their valuable remarks.\newline

\nopagebreak
\vskip.5cm
\end{document}